\documentclass[11pt,a4paper]{article}

\usepackage{amsmath}
\usepackage{amssymb,verbatim}
\usepackage{amsfonts}
\usepackage{amsthm}
\usepackage{pifont}
\usepackage{graphicx}

\newcommand{\beq}{\begin{equation*}}
\newcommand{\eeq}{\end{equation*}}
\newcommand{\beqn}{\begin{equation}}
\newcommand{\eeqn}{\end{equation}}
\newcommand{\ba}{\begin{array}}
\newcommand{\ea}{\end{array}}
\def\bal#1\eal{\begin{align*}\begin{split}#1\end{split}\end{align*}}
\def\baln#1\ealn{\begin{align}\begin{split}#1\end{split}\end{align}}
\newcommand{\ds}{\displaystyle}
\newcommand{\ts}{\textstyle}

\newcommand{\Z}{\mathbb Z}

\newcommand{\R}{\mathcal{R}_{a,\,b,\,\alpha}}
\newcommand{\Ra}{\mathcal{R}_a}
\newcommand{\Rb}{\mathcal{R}_b}
\newcommand{\Sn}{\mathcal{S}_{n,\,\ell}}
\newcommand{\dd}{\mathrm{d}}
\newcommand{\F}{\mathcal{F}}

\theoremstyle{theorem}
\newtheorem{satz}{Theorem}
\newtheorem{coro}{Corollary}

\begin{document}

\title{Buffon's problem with a star of needles \\ and a lattice of parallelograms \vspace{-0.1cm}}
\author{Uwe B\"asel}
\date{}
\maketitle
\vspace{-0.7cm}

\begin{abstract}
\noindent
A star of $n\geq 2$ line segments (needles) of equal length with common endpoint and constant angular spacing is randomly placed onto a lattice which is the union of two families of equidistant lines in the plane with angle $\alpha$ between the nonparallel lines.
For odd $n$, we calculate the probabilities of exactly $i$ intersections between the star and the lattice (for even $n$, see \cite{Baesel3}). Using a geometrical method, we derive the limit distribution function of the relative number of intersections as $n\rightarrow\infty$. This function is independent of~$\alpha$. We show that the relative numbers for each of the two families are asymptotically independent random variables.\\[0.2cm]  
\vspace{-0.1cm}
\end{abstract}
\section{Introduction}
\vspace{-0.2cm}
We consider the random throw of a star $\Sn$ of line segments onto a plane ruled with two families $\Ra$ and $\Rb$ of parallel lines,
\begin{align*}
  \Ra & := \{(x,y)\in\mathbb{R}^2\mid x\sin\alpha-y\cos\alpha=ka\,,k\in\Z\}\,,\\
  \Rb & := \{(x,y)\in\mathbb{R}^2\mid y=mb\,,m\in\Z\}\,,
\end{align*}
where $a$ and $b$ are positive real constants, $\alpha\in\mathbb{R}$, $0<\alpha\leq\pi/2$, and put $\R:=\Ra\cup\Rb$. We denote the parallelogram
\[ \F:=\{(x,y)\in\mathbb{R}^2\;|\;0\leq y\leq b\,,\;y\cot\alpha\leq x\leq a\csc\alpha+y\cot\alpha\} \]  
shown in $\mbox{Fig.}\,\ref{Bild1}$ the {\em fundamental cell of $\R$}. The star $\Sn$ consists of $n$ ($2\leq n<\infty$) line segments ({\em needles}) of equal length $\ell$ with common endpoint and constant angular spacing $2\pi/n$ between neighbouring needles. (The convex hull of $\Sn$ is the regular $n$-gon with circumscribed circle of radius $\ell$.) 

\begin{figure}[h]
  \vspace{-0.4cm}
  \begin{center}
    \includegraphics[scale=0.85]{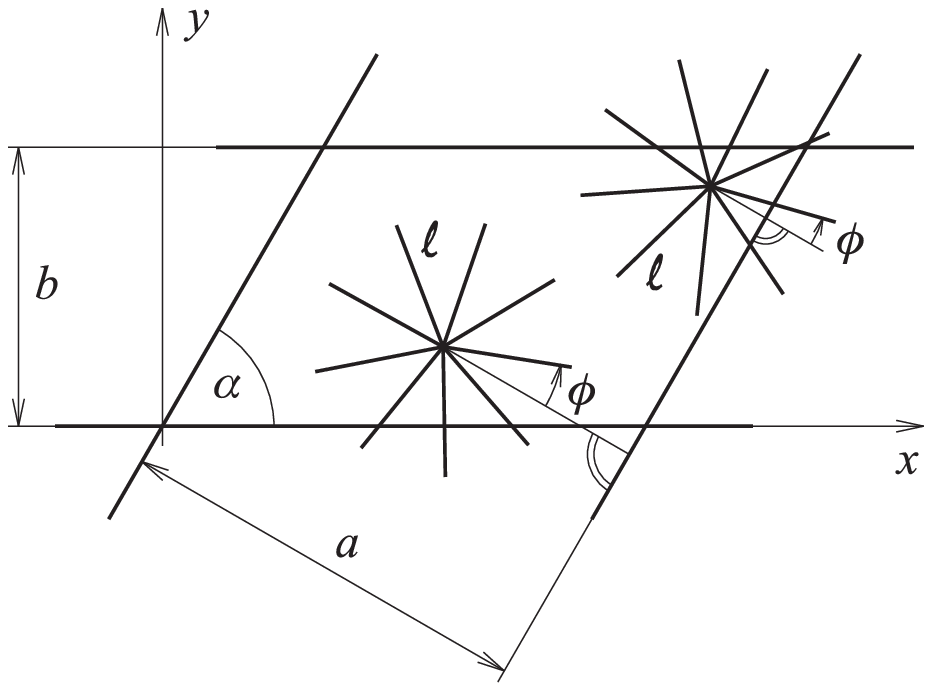}
  \end{center}
  \vspace{-0.7cm}
  \caption{\label{Bild1} Star $\Sn$ (Example $n=9$) and lattice and lattice $\R$}
  \vspace{-0.4cm}
\end{figure}

The {\em random throw $\Sn$ of onto $\R$} is defined as follows: The coordinates $x$ and $y$ of the centre point of $\Sn$ are random variables uniformly distributed in $[y\cot\alpha,a\csc\alpha+y\cot\alpha]$ and $[0,b]$ resp.; the angle $\phi$ between the direction perpendicular to the lines of $\Ra$ and a certain needle of $\Sn$ is a random variable uniformly distributed in $[0,2\pi]$. All 3 random variables are stochastically independent. We assume $2\ell\sin(\frac{\pi}{n}\lfloor\frac{n}{2}\rfloor)\leq\min(a,b)$; in this case the probability that $\Sn$ intersects two lines of $\Ra$ (or $\Rb$) at the same time is equal to zero. The maximum number $M$ of intersections with $\Ra$ (or $\Rb$) is then given by
\beq
  M = \left\{\begin{array}{cl}
    n/2\,, & \mbox{if\;\;\:$n$ is even}\,,\\
	(n+1)/2\,, & \mbox{if\;\;\:$n$ is odd}\,.   
  \end{array}\right.
\eeq

In \cite{Buffon}, Buffon published the solution of his famous needle problem. It is the calculation of the probability of the event	 that $\mathcal{S}_{2,\,\ell}$ intersects $\Ra$. ($\mathcal{S}_{2,\,\ell}$ can be considered as single needle of length $2\ell$.) 
Laplace \cite[pp.\ 359-362]{Laplace} calculated the intersection probability for $\mathcal{S}_{2,\,\ell}$ and $\mathcal{R}_{a,\,b,\,\pi/2}$. Santal\'o \cite{Santalo1} generalized this result for $\R$, $0<\alpha\leq\pi/2$, and derived the probabilites of 0, 1 or 2 intersection points (see also \cite[p.\ 139]{Santalo2}). Duma and Stoka \cite{Duma_Stoka} solved the problem for ellipses and $\mathcal{R}_{a,\,b,\,\pi/2}$. Ren and Zhang \cite{Ren_Zhang} and Aleman et al.\;\cite{Aleman} calculated the intersection probability for an arbitrary convex body $K$ and $\R$, and proved that for $K$ there is an nonvanishing value of $\alpha$ for which the events {\em $K$ intersects $\Ra$} and {\em $K$ intersects $\Rb$} are independent; explicit results for regular $n$-gons ($n\geq 2$) and $\R$ were obtained by B\"asel~\cite{Baesel4}. In \cite{Baesel3}, B\"asel calculated the probabilities of exactly $i$ intersections for $\Sn$ with even $n\geq 2$ and $\R$. Bonanzinga \cite{Bonanzinga} found the intersection probabilities for $\mathcal{S}_{3,\,\ell}$ and $\R$, $\pi/3\leq\alpha\leq\pi/2$.

In the Sections 2 and 3 we calculate the probabilities of exactly $i$ intersections for $\Sn$ with odd $n$, $3\leq n<\infty$, and $\R$, $0<\alpha\leq\pi/2$. In Section~4 we investigate the distribution functions of the {\em relative number of intersections} for $n\in\mathbb{N}$, $n\geq 2$. Using a geometrical method, we derive the limit distribution as $n\rightarrow\infty$. For abbreviation we put $\lambda=\ell/a$, $\mu=\ell/b$, and $\lfloor\,\cdot\,\rfloor$ for the integer part of $\,\cdot\,$.

\section{Intersection probabilities}

\begin{satz} \label{prob}
The probabilities $p(i)$ of exactly $i$ intersections between $\Sn$ and $\R$ are for odd $n\geq 3$, $2\max(\lambda,\mu)\sin(\frac{\pi}{n}\lfloor\frac{n}{2}\rfloor)\leq 1$ and $0<\alpha\leq\frac{\pi}{2n}$ given by
\beq
p(i) = \left\{
\begin{array}{l}
\begin{array}{ll}
  1-\Big[\frac{2n(\lambda+\mu)}{\pi}\,\sin\frac{\pi}{n}-\frac{n\lambda\mu}{\pi}\,f_0(\alpha)\Big]\;, &
					\mbox{if $i=0$}\,,\\[0.25cm]
  \frac{8n(\lambda+\mu)}{\pi}\,\sin^2\frac{\pi}{2n}\sin\frac{i\pi}{n}-\frac{4n\lambda\mu}{\pi}
					\Big[f_1(\alpha)\sin\frac{i\pi}{n}\\[0.25cm]
  \quad -\,f_4(\alpha)\Big(\!\cot\frac{\pi}{n}\sin\frac{i\pi}{n}-i\cos\frac{i\pi}{n}\Big)\Big]\;, &
					\mbox{if $1\leq i\leq M-2$}\,,\\[0.25cm]
  \frac{4n(\lambda+\mu)}{\pi}\Big(\!\cos\frac{\pi}{2n}-\cos^2\frac{3\pi}{4n}\Big)-\frac{2n\lambda\mu}{\pi}\Big[f_2(\alpha)\\[0.25cm]
  \quad -\,2f_4(\alpha)\Big(\!\cot\frac{\pi}{n}\sin\frac{i\pi}{n}-i\cos\frac{i\pi}{n}\Big)\Big]\;, &
					\mbox{if $i=M-1$}\,,\\[0.25cm]
  \frac{4n(\lambda+\mu)}{\pi}\sin^2\frac{\pi}{4n}-\frac{n\lambda\mu}{2\pi}\,\big[4f_3(\alpha)-f_7(\alpha)\big]\;, & 
					\mbox{if $i=M$ and $n=3$}\,,\\[0.25cm]
\end{array}\\
\begin{array}{l}					
  \frac{4n(\lambda+\mu)}{\pi}\sin^2\frac{\pi}{4n}-\frac{2n\lambda\mu}{\pi}\Big[f_3(\alpha)-4f_5(\alpha)\sin\frac{\pi}{n}\\[0.25cm]
  \;\; -\,f_4(\alpha)\Big\{\!(n-5)\sin\frac{\pi}{2n}+2\csc\frac{\pi}{n}\cos\frac{5\pi}{2n}\Big\}\Big]\;,\quad 
					\mbox{if $i=M$ and $n\geq 5$}\,,\\[0.25cm]
  \frac{4n\lambda\mu}{\pi}\Big[2f_6(\alpha)\sin\frac{(i-M)\pi}{n}+2f_5(\alpha)\sin\frac{(i+1-M)\pi}{n}\\[0.25cm]
  \quad -\,f_4(\alpha)\Big\{\!\big(2M-i-3\big)\cos\frac{i\pi}{n}\\[0.25cm]
\end{array}\\
\begin{array}{ll}  
  \quad -\,\csc\frac{\pi}{n}\sin\frac{(2M-i-3)\pi}{n}\Big\}\Big]\;, & \mbox{if $M+1\leq i\leq 2M-3$}\,,\\[0.25cm]		
  \frac{n\lambda\mu}{2\pi}\Big[16f_6(\alpha)\cos\frac{3\pi}{2n}+f_7(\alpha)\Big]\;, &
					\mbox{if $i=2M-2$ and $n\geq 5$}\,,\\[0.25cm]
\end{array}\\
\begin{array}{ll}					
  \frac{n\lambda\mu}{\pi}\,f_8(\alpha)\,, & \mbox{if $i=2M-1$}\,,\\[0.25cm]
  \frac{n\lambda\mu}{2\pi}\,f_9(\alpha)\,, & \mbox{if $i=2M$}\,,
\end{array} 
\end{array}
\right.
\eeq
where
\begin{align*}
  f_0(\alpha) = {} & {\ts 2\big[\frac{\pi}{n}\cos\alpha+g\big(\frac{\pi}{n}-\alpha\big)+h(\alpha)\big]\cos^2\frac{\pi}{2n}\,}\,,\\[0.1cm]
  f_1(\alpha) = {} & {\ts \big[\frac{\pi}{n}\cos\alpha+g\big(\frac{\pi}{n}-\alpha\big)+h(\alpha)\big]\sin\frac{\pi}{n}}\,,
		\displaybreak[0]\\[0.1cm] 
  f_2(\alpha) = {} & {\ts \big[\frac{2\pi}{n}\cos\frac{\pi}{2n}\cos\alpha+g\big(\frac{3\pi}{2n}-\alpha\big)
		-g\big(\frac{\pi}{2n}-\alpha\big)+h\big(\frac{\pi}{2n}+\alpha\big)\big]\sin\frac{\pi}{n}\,}\,,\displaybreak[0]\\[0.1cm]
  f_3(\alpha) = {} & {\ts g\big(\frac{\pi}{2n}-\alpha\big)\sin\frac{\pi}{n}}\,,\displaybreak[0]\\[0.1cm]
  f_4(\alpha) = {} & {\ts \big[\frac{\pi}{n}\cos\alpha+g\big(\frac{\pi}{n}-\alpha\big)+h(\alpha)\big]\sin^2\frac{\pi}{2n}}\,,
		\displaybreak[0]\\[0.1cm]
  f_5(\alpha) = {} & {\ts \big[\frac{2\pi}{n}\cos\frac{\pi}{2n}\cos\alpha+g\big(\frac{3\pi}{2n}-\alpha\big)-g\big(\frac{\pi}{2n}-\alpha\big)
		+h\big(\frac{\pi}{2n}+\alpha\big)\big]\sin^2\frac{\pi}{2n}}\,,\displaybreak[0]\\[0.1cm]
  f_6(\alpha) = {} & {\ts g\big(\frac{\pi}{2n}-\alpha\big)\sin^2\frac{\pi}{2n}}\,,\displaybreak[0]\\[0.1cm]
  f_7(\alpha) = {} & {\ts \frac{\pi}{n}\big(3-2\cos\frac{2\pi}{n}\big)\cos\alpha-g\big(\frac{3\pi}{n}-\alpha\big)
		+3g\big(\frac{2\pi}{n}-\alpha\big)-g\big(\frac{\pi}{n}-\alpha\big)}\\
				   & {\ts +\,7h(\alpha)-h\big(\frac{\pi}{n}+\alpha\big)-h\big(\frac{2\pi}{n}+\alpha\big)}\,\displaybreak[0]\\[0.1cm]
  f_8(\alpha) = {} & {\ts -\frac{\pi}{n}\cos\alpha-g\big(\frac{2\pi}{n}-\alpha\big)+2g\big(\frac{\pi}{n}-\alpha\big)-4h(\alpha)
		+h\big(\frac{\pi}{n}+\alpha\big)}\,,\displaybreak[0]\\[0.1cm]
  f_9(\alpha) = {} & {\ts \frac{\pi}{n}\cos\alpha-g\big(\frac{\pi}{n}-\alpha\big)+3h(\alpha)}	
\end{align*}
with
\beq
  g(x)=\sin x+\alpha\cos x \quad\mbox{and}\quad h(x)=\sin x-\alpha\cos x\,.
\eeq
\end{satz}

\begin{figure}[h]
  \vspace{0cm}
  \begin{center}
    \includegraphics[scale=0.7]{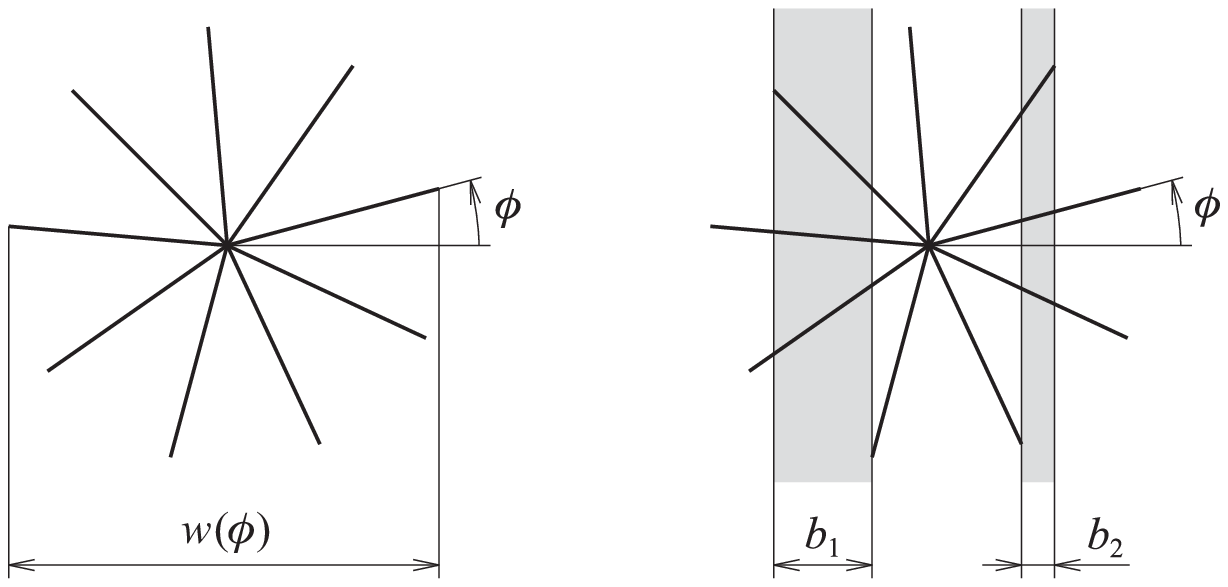}
  \end{center}
  \vspace{-0.6cm}
  \caption{\label{Bild2} $w(\phi)$ and stripes of $s(3,\phi)$ for $\mathcal{S}_{9,\,\ell}$}
  \vspace{0.3cm}
  \begin{center}
    \includegraphics[scale=0.85]{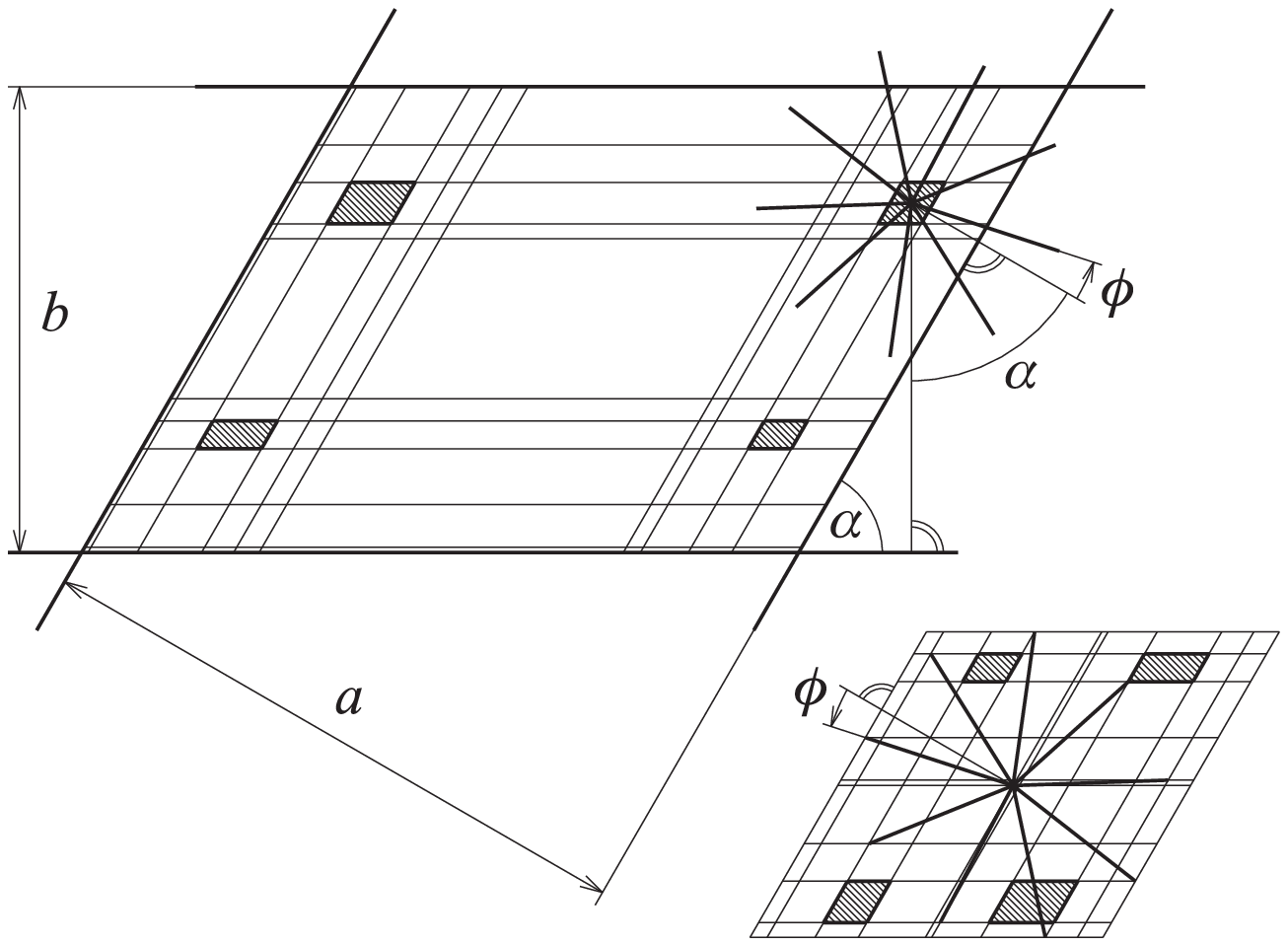}
  \end{center}
  \vspace{-0.6cm}
  \caption{\label{Bild3} $E_{3,\,2}$ for star $\mathcal{S}_{9,\,\ell}$, lattice $\R$ and fixed value of $\phi$}
  \vspace{-0.1cm}
\end{figure}

\begin{proof}
We denote by $w(\phi)$ the width of $\Sn$ (with angle $\phi$) perpendicular to the lines of $\Ra$ (see Fig.\ \ref{Bild2} (left side) and Fig.\ \ref{Bild3}), and by $s(k,\phi)$ the {\em breadth functions} of exactly $k$, $k\in\{1,2,\ldots,M\}$, intersections between $\Sn$ (with angle $\phi$) and $\Ra$. $s(k,\phi)$ is the breadth of one stripe or the sum of the breadths of two stripes. An example of $s(3,\phi)$ for $\mathcal{S}_{9,\,\ell}$ is shown on the right side of Fig.\ \ref{Bild2}. Here, $s(3,\phi)$ is the sum of the breadths $b_1=b_1(\phi)$ and $b_2=b_2(\phi)$. From the symmetry of $\Sn$, it follows that $w$ and $s(k,\,\cdot\,)$ are $\pi/n$-periodic functions. In the following, we have to consider these functions in the half-open intervals:
\[ \mathcal{I}_1:=\bigg[0\,,\,\frac{\pi}{2n}\bigg) \;,\quad \mathcal{I}_2:=\bigg[\frac{\pi}{2n}\,,\,\frac{\pi}{n}\bigg)
	\quad\mbox{and}\quad \mathcal{I}_3:=\bigg[\frac{\pi}{n}\,,\,\frac{3\pi}{2n}\bigg)\,. \]
The required restrictions of the function $w$ are given by
\beq
\begin{array}{c@{\;:=\;}c@{\;=\;}l}
  w_{12}(\phi) & w|_{\mathcal{I}_1\cup\,\mathcal{I}_2}(\phi) & 2\ell\cos\frac{\pi}{2n}\cos\big(\phi-\frac{\pi}{2n}\big)\,,\\[0.15cm]
  w_3(\phi) & w|_{\mathcal{I}_3}(\phi) & 2\ell\cos\frac{\pi}{2n}\cos\big(\phi-\frac{3\pi}{2n}\big)\,.
\end{array}
\eeq
For $s(k,\,\cdot\,)$ and $1\leq k\leq M-2$, one finds
\beq
\begin{array}{c@{\;:=\;}c@{\;=\;}l}
  s_{12}(k,\phi) & s|_{\mathcal{I}_1\cup\,\mathcal{I}_2}(k,\phi) & 4\ell\sin\frac{k\pi}{n}\sin\frac{\pi}{2n}
					\cos\big(\phi-\frac{\pi}{2n}\big)\,,\\[0.15cm]
  s_3(k,\phi) & s|_{\mathcal{I}_3}(k,\phi) & 4\ell\sin\frac{k\pi}{n}\sin\frac{\pi}{2n}\cos\big(\phi-\frac{3\pi}{2n}\big)\,,
\end{array}
\eeq
for $k=M-1$,
\beq
\begin{array}{c@{\;:=\;}c@{\;=\;}l}
  s_1(k,\phi) & s|_{\mathcal{I}_1}(k,\phi) & \ell\big[2\cos\frac{\pi}{2n}\sin\phi-\sin\big(\phi-\frac{3\pi}{2n}\big)\big]\,,\\[0.15cm]
  s_2(k,\phi) & s|_{\mathcal{I}_2}(k,\phi) & \ell\big[\!-\!2\cos\frac{\pi}{2n}\sin\big(\phi-\frac{\pi}{n}\big)
					+\sin\big(\phi+\frac{\pi}{2n}\big)\big]\,,\\[0.15cm]
  s_3(k,\phi) & s|_{\mathcal{I}_3}(k,\phi) & \ell\big[2\cos\frac{\pi}{2n}\sin\big(\phi-\frac{\pi}{n}\big)
					-\sin\big(\phi-\frac{5\pi}{2n}\big)\big]\,,
\end{array}
\eeq
and for $k=M$,
\beq
\begin{array}{c@{\;=\;}c@{\;=\;}l}
  s_1(k,\phi) & s|_{\mathcal{I}_1}(k,\phi) & -\ell\sin\big(\phi-\frac{\pi}{2n}\big)\,,\\[0.15cm]
  s_2(k,\phi) & s|_{\mathcal{I}_2}(k,\phi) & \ell\sin\big(\phi-\frac{\pi}{2n}\big)\,,\\[0.15cm]
  s_3(k,\phi) & s|_{\mathcal{I}_3}(k,\phi) & -\ell\sin\big(\phi-\frac{3\pi}{2n}\big)\,.
\end{array}
\eeq
$E_{k,\,m}$, $0\leq k,m<M$, denotes the event that $\Sn$ has exactly $k$ intersections with $\Ra$ and (at the same time) exactly $m$ intersections with $\Rb$. For fixed value of $\phi$, this event occurs if the centre point of $\Sn$ is in one, two or four disjunct parallelograms that are subsets of $\F$. (An example is shown in Fig.~\ref{Bild3}. For the given angle $\phi$, the event $E_{3,\,2}$ occurs if the centre point of $\mathcal{S}_{9,\,\ell}$ is in one of the four hatched parallelograms.) 
$s(k,\phi)\,s(m,\phi+\alpha)/\sin\alpha$ is the area of the one parallelogram or the sum of the areas of the two or four parallelograms if $1\leq k,m<M$. Therefore, the conditional probability of the event $E_{k,\,m}$ for fixed angle $\phi$ is given by
\beq
  P(E_{k,\,m}\,|\,\phi)
	= \frac{s(k,\phi)\,s(m,\phi+\alpha)/\sin\alpha}{\mbox{Area}\;\mathcal{F}}
	= \frac{1}{ab}\,s(k,\phi)\,s(m,\phi+\alpha)\,.
\eeq
For $0\leq k,m<M$, we have
\begin{align*}
  P(E_{0,\,0}\,|\,\phi)\, = {} & \frac{1}{ab}\,[a-w(\phi)]\,[b-w(\phi+\alpha)]\,,\displaybreak[0]\\[0.05cm]
  P(E_{0,\,m}\,|\,\phi) = {} & \frac{1}{ab}\,[a-w(\phi)]\,s(m,\phi+\alpha)\,,\displaybreak[0]\\[0.05cm]
  P(E_{k,\,0}\,|\,\phi)\, = {} & \frac{1}{ab}\,s(k,\phi)\,[b-w(\phi+\alpha)]\,.
\end{align*}
The density function of the random variable $\phi$ is given by
\beq
f(\phi) = \left\{\begin{array}{c@{\quad\mbox{if}\quad}l}
  \dfrac{n}{\pi} &  \phi\in\mathcal{I}_1\cup\mathcal{I}_2\,,\\[0.3cm]
  0 & \phi\in\mathbb{R}\setminus\mathcal{I}_1\cup\mathcal{I}_2\,.
\end{array}\right.  
\eeq
Therefore, the (total) probability of the event $E_{k,\,m}$ ist given by 
\beq
  P(E_{k,\,m}) 
	= \int_0^{\pi/n}P(E_{k,\,m}\,|\,\phi)\,f(\phi)\,\dd\phi 
	= \frac{n}{\pi}\int_0^{\pi/n}P(E_{k,\,m}\,|\,\phi)\,\dd\phi\,.
\eeq
From the piecewise definition of the functions $w$ and $s(k,\,\cdot\,)$, it follows that we have to distinguish (in general) the cases
\[ 0\leq\phi<\frac{\pi}{2n}-\alpha \;,\;\; \frac{\pi}{2n}-\alpha\leq\phi<\frac{\pi}{2n} \;,\;\;
   \frac{\pi}{2n}\leq\phi<\frac{\pi}{n}-\alpha \;,\;\; \frac{\pi}{n}-\alpha\leq\phi<\frac{\pi}{n}\,. \]
We calculate the probabilities $P(E_{k,\,m})$ in some examples. For $k=0$ and $1\leq m\leq M-2$, we get
\begin{align*}
  P(E_{0,\,m})
	= {} & \frac{n}{\pi ab}\bigg(\int_0^{\pi/n-\alpha}+\int_{\pi/n-\alpha}^{\pi/n}\bigg)\:[a-w(\phi)]\,s(m,\phi+\alpha)\,\dd\phi\\
	= {} & \frac{n}{\pi ab}\,\bigg(\int_0^{\pi/n-\alpha}[a-w_{12}(\phi)]\,s_{12}(m,\phi+\alpha)\,\dd\phi\\
		 & \qquad\; +\int_{\pi/n-\alpha}^{\pi/n}[a-w_{12}(\phi)]\,s_3(m,\phi+\alpha)\,\dd\phi\bigg)\\
	= {} & \frac{8n\mu}{\pi}\sin^2\frac{\pi}{2n}\sin\frac{m\pi}{n}-\frac{2n\lambda\mu}{\pi}\sin\frac{m\pi}{n}\,f_1(\alpha)\,.
\end{align*}
Due to symmetry, for $1\leq k\leq M-2$ and $m=0$, we get
\begin{align*}
  P(E_{k,\,0}) = {} & \frac{8n\lambda}{\pi}\sin^2\frac{\pi}{2n}\sin\frac{k\pi}{n}-\frac{2n\lambda\mu}{\pi}\sin\frac{k\pi}{n}
		\,f_1(\alpha)\,.
\end{align*}
For $1\leq k,m\leq M-2$, we find
\begin{align*}
  P(E_{k,\,m}) 
	= {} & \frac{n}{\pi ab}\bigg(\int_0^{\pi/n-\alpha}+\int_{\pi/n-\alpha}^{\pi/n}\bigg)\:s(k,\phi)\,s(m,\phi+\alpha)\,\dd\phi\\
	= {} & \frac{8n\lambda\mu}{\pi}\sin\frac{k\pi}{n}\sin\frac{m\pi}{n}\,f_4(\alpha)\,,
\end{align*}
for $k=M-1$ and $m=M$,
\begin{align*}
  P(E_{M-1,\,M}) 
	= {} & \frac{n}{\pi ab}\bigg(\int_0^{\frac{\pi}{2n}-\alpha}+\int_{\frac{\pi}{2n}-\alpha}^\frac{\pi}{2n}
			+\int_\frac{\pi}{2n}^{\frac{\pi}{n}-\alpha}+\int_{\frac{\pi}{n}-\alpha}^\frac{\pi}{n}\bigg)\:s(M-1,\phi)\\
		 & \times s(M,\phi+\alpha)\,\dd\phi = \frac{n\lambda\mu}{2\pi}\,f_8(\alpha)\,,
\end{align*}
and due to symmetry, $P(E_{M,\,M-1})=P(E_{M-1,\,M})$.\\[0.2cm]
The remaining calculations deliver the results
\begin{align*}
P(E_{k,\,m}) = {} & 1-\bigg(\frac{2n(\lambda+\mu)}{\pi}\,\sin\frac{\pi}{n}-\frac{n\lambda\mu}{\pi}\,f_0(\alpha)\bigg)\,,\;
\mbox{($k=0=m$)}\,,\\
P(E_{k,\,m}) = {} & \frac{4n\mu}{\pi}\bigg(\cos\frac{\pi}{2n}-\cos^2\frac{3\pi}{4n}\bigg)-\frac{n\lambda\mu}{\pi}\,f_2(\alpha)\,,\;
\mbox{($k=0$, $m=M-1$})\,,\displaybreak[0]\\
P(E_{k,\,m}) = {} & \frac{4n\lambda}{\pi}\bigg(\cos\frac{\pi}{2n}-\cos^2\frac{3\pi}{4n}\bigg)-\frac{n\lambda\mu}{\pi}\,f_2(\alpha)\,,\;
\mbox{($k=M-1$, $m=0$})\,,\displaybreak[0]\\
P(E_{k,\,m}) = {} & \frac{4n\mu}{\pi}\sin^2\frac{\pi}{4n}-\frac{n\lambda\mu}{\pi}\,f_3(\alpha)\,,\;\mbox{($k=0$, $m=M$)}\,,\\
P(E_{k,\,m}) = {} & \frac{4n\lambda}{\pi}\sin^2\frac{\pi}{4n}-\frac{n\lambda\mu}{\pi}\,f_3(\alpha)\,,\;\mbox{($k=M$, $m=0$)}\,,
	\displaybreak[0]\\
P(E_{k,\,m}) = {} & \frac{4n\lambda\mu}{\pi}\sin\frac{k\pi}{n}\,f_5(\alpha)\,,\;\mbox{($1\leq k\leq M-2$, $m=M-1$)}\,,\\
P(E_{k,\,m}) = {} & \frac{4n\lambda\mu}{\pi}\sin\frac{m\pi}{n}\,f_5(\alpha)\,,\;\mbox{($k=M-1$, $1\leq m\leq M-2$)}\,,\\  
P(E_{k,\,m}) = {} & \frac{4n\lambda\mu}{\pi}\sin\frac{k\pi}{n}\,f_6(\alpha)\,,\;\mbox{($1\leq k\leq M-2$, $m=M$)}\,,\\
P(E_{k,\,m}) = {} & \frac{4n\lambda\mu}{\pi}\sin\frac{m\pi}{n}\,f_6(\alpha)\,,\;\mbox{($k=M$, $1\leq m\leq M-2$)}\,,\\
P(E_{k,\,m}) = {} & \frac{n\lambda\mu}{2\pi}\,f_7(\alpha)\,,\;\mbox{($k=M-1=m$)}\,,\\
P(E_{k,\,m}) = {} & \frac{n\lambda\mu}{2\pi}\,f_9(\alpha)\,,\;\mbox{($k=M=m$)}\,.
\end{align*}
The probabilities $p(i)$ of exactly $i$ intersections between $\Sn$ und $\R$ are given by
\beq p(i) = \left\{\ba{cll}
  \ds{\sum_{k=0}^i\;P(E_{k,\,i-k})} & \mbox{for} & 0\leq i\leq M\;,\\[0.5cm]
  \ds{\sum_{k=i-M}^M P(E_{k,\,i-k})} & \mbox{for} & M+1\leq i\leq 2M\;.	
\ea\right.
\eeq
We have
\[ p(0) = P(E_{0,\,0}) = 1-\bigg(\frac{2n(\lambda+\mu)}{\pi}\,\sin\frac{\pi}{n}-\frac{n\lambda\mu}{\pi}\,f_0(\alpha)\bigg)\,.\]
For $1\leq i\leq M-2$, one finds
\begin{align*}
  p(i) 
	= {} & P(E_{0,\,i}) + P(E_{i,\,0}) + \sum_{k=1}^{i-1}\;P(E_{k,\,i-k})\\
	= {} & \frac{8n(\lambda+\mu)}{\pi}\sin^2\frac{\pi}{2n}\sin\frac{i\pi}{n}-\frac{4n\lambda\mu}{\pi}\,f_1(\alpha)\sin\frac{i\pi}{n}\\
		 & +\frac{8n\lambda\mu}{\pi}\,f_4(\alpha)\,\sum_{k=1}^{i-1}\sin\frac{k\pi}{n}\sin\frac{(i-k)\pi}{n}
\end{align*}
with
\[ \sum_{k=1}^{i-1}\sin\frac{k\pi}{n}\sin\frac{(i-k)\pi}{n}
	= \frac{1}{2}\,\bigg(\cot\frac{\pi}{n}\sin\frac{i\pi}{n}-i\cos\frac{i\pi}{n}\bigg)\,. \]
For $i=M-1$, we get
\begin{align*}
  p(i) 
	= {} & P(E_{0,\,M-1}) + P(E_{M-1,\,0}) + \sum_{k=1}^{M-2}\;P(E_{k,\,i-k})\displaybreak[0]\\
	= {} & \frac{4n(\lambda+\mu)}{\pi}\,\bigg(\cos\frac{\pi}{2n}-\cos^2\frac{3\pi}{4n}\bigg)-\frac{2n\lambda\mu}{\pi}\,f_2(\alpha)\\
		 & +\frac{8n\lambda\mu}{\pi}\,f_4(\alpha)\,\sum_{k=1}^{i-1}\sin\frac{k\pi}{n}\sin\frac{(i-k)\pi}{n}
\end{align*}
with the sum as above. For $i=M=2$ and $n=3$, we find
\begin{align*}
  p(M) 
	= {} & P(E_{0,\,M}) + P(E_{M,\,0}) + P(E_{M-1,\,M-1})\\
	= {} & \frac{4n(\lambda+\mu)}{\pi}\sin^2\frac{\pi}{4n}-\frac{2n\lambda\mu}{\pi}\,f_3(\alpha)
			+\frac{n\lambda\mu}{2\pi}\,f_7(\alpha)\sin\frac{\pi}{n}\,,
\end{align*}
and for $i=M$ and $n\geq 5$,
\begin{align*}
  p(M) 
	= {} & P(E_{0,\,M}) + P(E_{M,\,0}) + P(E_{1,\,M-1}) + P(E_{M-1,\,1}) + \sum_{k=2}^{M-2}\,P(E_{k,\,i-k})\\
	= {} & \frac{4n(\lambda+\mu)}{\pi}\sin^2\frac{\pi}{4n}-\frac{2n\lambda\mu}{\pi}\,f_3(\alpha)
				+\frac{8n\lambda\mu}{\pi}\,f_5(\alpha)\sin\frac{\pi}{n}\\
		 & +\frac{8n\lambda\mu}{\pi}\,f_4(\alpha)\,\sum_{k=2}^{M-2}\sin\frac{k\pi}{n}\sin\frac{(M-k)\pi}{n}	
\end{align*}
with
\begin{align*}
  \sum_{k=2}^{M-2}\sin\frac{k\pi}{n}\sin\frac{(M-k)\pi}{n}
	= {} & \frac{1}{2}\,\bigg(\!-\!(M-3)\cos\frac{M\pi}{n}+\csc\frac{\pi}{n}\sin\frac{(M-3)\pi}{n}\bigg)\\
	= {} & \frac{1}{4}\,\bigg((n-5)\sin\frac{\pi}{2n}+2\csc\frac{\pi}{n}\cos\frac{5\pi}{2n}\bigg)\,.
\end{align*}
For the case $M+1\leq i\leq 2M-3$, we put $i=2M-\nu$. So we have to consider all $\nu$ with $3\leq\nu\leq M-1$. One finds
\begin{align*}
  p(2M-\nu)
	= {} & \sum_{k=(2M-\nu)-M}^M P(E_{k,\,2M-\nu-k}) = \sum_{k=M-\nu}^M P(E_{k,\,2M-\nu-k})\displaybreak[0]\\[0.2cm]
	= {} & P(E_{M-\nu,\,M}) + P(E_{M,\,M-\nu}) + P\big(E_{M-(\nu-1),\,M-1}\big)\\
		 & + P\big(E_{M-1,\,M-(\nu-1)}\big) + \sum_{k=M-(\nu-2)}^{M-2} P(E_{k,\,2M-\nu-k})\displaybreak[0]\\[0.2cm]
	= {} & \frac{8n\lambda\mu}{\pi}\,f_6(\alpha)\sin\frac{(M-\nu)\pi}{n}
				+ \frac{8n\lambda\mu}{\pi}\,f_5(\alpha)\sin\frac{[M-(\nu-1)]\pi}{n}\\
		 &		+ \frac{8n\lambda\mu}{\pi}\,f_4(\alpha)\,\sum_{k=M-(\nu-2)}^{M-2}\sin\frac{k\pi}{n}\sin\frac{(2M-\nu-k)\pi}{n} 
\end{align*}
with
\beq
 \begin{array}{l}
  \ds{\sum_{k=M-(\nu-2)}^{M-2}\sin\frac{k\pi}{n}\sin\frac{(2M-\nu-k)\pi}{n}}\\[0.55cm]
  \ds{\quad = \frac{1}{2}\bigg(\!-\!(\nu-3)\cos\frac{(2M-\nu)\pi}{n}+\csc\frac{\pi}{n}\sin\frac{(\nu-3)\pi}{n}}\bigg)\,,
 \end{array}
\eeq
and therefore, with $\nu=2M-i$,
\begin{align*}
  p(i)
	= {} & \frac{4n\lambda\mu}{\pi}\bigg[2f_6(\alpha)\sin\frac{(i-M)\pi}{n}+2f_5(\alpha)\sin\frac{(i+1-M)\pi}{n}\\
		 & -f_4(\alpha)\bigg(\!\big(2M-i-3\big)\cos\frac{i\pi}{n}-\csc\frac{\pi}{n}\sin\frac{(2M-i-3)\pi}{n}\bigg)\bigg]\,.
\end{align*}
For $i=2M-2$ and $n\geq 5$, we get
\begin{align*}
  p(2M-2)
	= {} & \sum_{k=M-2}^M P\big(E_{k,\,(2M-2)-k}\big)\\
	= {} & P(E_{M-2,\,M}) + P(E_{M,\,M-2}) + P(E_{M-1,\,M-1})\\
	= {} & \frac{8n\lambda\mu}{\pi}\,f_6(\alpha)\sin\frac{(M-2)\pi}{n} + \frac{n\lambda\mu}{2\pi}\,f_7(\alpha)\\
	= {} & \frac{n\lambda\mu}{2\pi}\bigg(16f_6(\alpha)\cos\frac{3\pi}{2n}+f_7(\alpha)\bigg)\,.
\end{align*}
Furthermore, we find
\begin{align*}
  p(2M-1)
	= {} & \sum_{k=M-1}^M P\big(E_{k,\,(2M-1)-k}\big) = P(E_{M-1,\,M}) + P(E_{M,\,M-1})\\
	= {} & \frac{n\lambda\mu}{\pi}\,f_8(\alpha)
\end{align*}
and finally
\[ p(2M) = P(E_{M,\,M}) = \frac{n\lambda\mu}{2\pi}\,f_9(\alpha)\,. \]
So, the proof is complete.
\end{proof} 

In the following, we write $p(i,\alpha)$ instead of $p(i)$ and $P_\alpha(E_{k,\,m})$ instead of $P(E_{k,\,m})$ to emphasize the dependence on $\alpha$.

\begin{satz} \label{Satz2}
For fixed values of odd $n\geq 3$, $a$, $b$ and $\ell$, the function
\[ p(i,\,\cdot\,) \,:\;\; [0,\pi/2] \;\rightarrow\; [0,1] \,,\quad \alpha\;\mapsto\; p(i,\alpha)  \]
is $\pi/n$-periodic. The restriction $p|_{[0,\,\pi/n]}$ is symmetric in relation to the line $\alpha=\pi/(2n)$.
\end{satz}
\begin{proof} The functions $w$ and $s(k,\,\cdot\,)$, $1\leq k\leq M$, are $\pi/n$-periodic. For $1\leq k,m\leq M$ and $\nu\in\mathbb{Z}$ we get
\begin{align*}
  P_{\alpha\,+\,\nu\pi/n}(E_{k,\,m})
	= {} & \frac{n}{\pi ab}\int_0^{\pi/n}s(k,\phi)\,s\Big(m,\phi+\alpha+\nu\,\frac{\pi}{n}\Big)\,\dd\phi\\
	= {} & \frac{n}{\pi ab}\int_0^{\pi/n}s(k,\phi)\,s(m,\phi+\alpha)\,\dd\phi
	= P_\alpha(E_{k,\,m})\,.
\end{align*}
This result holds for all values of $k$ and $m$, $0\leq k,m\leq M$. Since $p(i,\alpha)$ is a sum of $\pi/n$-periodic functions, it is $\pi/n$-periodic.

$s(k,\phi)\,s(m,\phi+\alpha)$ are $\pi/n$-periodic functions. Hence
\begin{align*}
  P_\alpha(E_{k,\,m})
	= {} & \frac{n}{\pi ab}\int_0^{\pi/n}s(k,\phi)\,s(m,\phi+\alpha)\,\dd\phi\\
	= {} & \frac{n}{\pi ab}\int_{-\alpha}^{\pi/n-\alpha}s(k,\phi)\,s(m,\phi+\alpha)\,\dd\phi\\
	= {} & \frac{n}{\pi ab}\int_0^{\pi/n}s(k,u-\alpha)\,s(m,u)\,\dd u\,,
\end{align*}
and therefore, with $\nu\in\mathbb{Z}$,
\begin{align*}
  P_{\nu\pi/n-\alpha}(E_{k,\,m})
	= {} & \frac{n}{\pi ab}\int_0^{\pi/n}s\Big(k,u-\Big(\nu\,\frac{\pi}{n}-\alpha\Big)\Big)\,s(m,u)\,\dd u\\
	= {} & \frac{n}{\pi ab}\int_0^{\pi/n}s(k,u+\alpha)\,s(m,u)\,\dd u
	= P_\alpha(E_{m,\,k})\,,
\end{align*}
For $1\leq k,m\leq M$, it follows that
\beq
  \begin{array}{l}
	P_{\nu\pi/n-\alpha}(E_{k,\,k}) = P_\alpha(E_{k,\,k})\,,\\[0.15cm] 
	P_{\nu\pi/n-\alpha}(E_{k,\,m})+P_{\nu\pi/n-\alpha}(E_{m,\,k}) = P_\alpha(E_{k,\,m})+P_\alpha(E_{m,\,k})\,.
  \end{array}
\eeq
Analogously, one gets
\beq
  \begin{array}{l}
	P_{\nu\pi/n-\alpha}(E_{0,\,0}) = P_\alpha(E_{0,\,0})\,,\\[0.15cm] 
	P_{\nu\pi/n-\alpha}(E_{0,\,m})+P_{\nu\pi/n-\alpha}(E_{m,\,0}) = P_\alpha(E_{0,\,m})+P_\alpha(E_{m,\,0})\,.
  \end{array}
\eeq
With $\nu=1$, we have
\beq
  \begin{array}{c@{\;=\;}c@{\;=\;}c@{\;=\;}c}
	p(0,\pi/n-\alpha)  & P_{\pi/n-\alpha}(E_{0,\,0}) & P_\alpha(E_{0,\,0}) & p(0,\alpha)\,,\\[0.15cm]
	p(2M,\pi/n-\alpha) & P_{\pi/n-\alpha}(E_{M,\,M}) & P_\alpha(E_{M,\,M}) & p(2M,\alpha)\,.
  \end{array}	
\eeq
For $1\leq i\leq 2M-1$, we find: If $i$ is odd, $p(i,\alpha)$ is the sum of terms $P_\alpha(E_{k,\,i-k})+P_\alpha(E_{i-k,\,k})$. If $i$ is even, $p(i,\alpha)$ is the sum of terms $P_\alpha(E_{k,\,i-k})+P_\alpha(E_{i-k,\,k})$ and one term $P(E_{i/2,\,i/2})$.

So, for every $i$, $0\leq i\leq 2M$, we have $p(i,\pi/n-\alpha)=p(i,\alpha)$; therefore, the restriction $p(i,\alpha)|_{[0,\,\pi/n]}$ is symmetric in relation to the line $\alpha=\pi/(2n)$.   
\end{proof}
\noindent From Theorem \ref{Satz2} one easily gets the following corollary:
\begin{coro}
The probabilities $p(i,\alpha)$ for $0<\alpha\leq\pi/2$ are given by
\beq
  p(i,\alpha) = \left\{\begin{array}{ll}
	p(i,\alpha-\delta(\alpha)) & \mbox{if $\;\,\alpha-\delta(\alpha)\leq\dfrac{\pi}{2n}$}\,,\\[0.3cm]
	p\bigg(i\,,\,\dfrac{\pi}{n}-\big[\alpha-\delta(\alpha)\big]\bigg) & \mbox{if $\;\,\alpha-\delta(\alpha)>\dfrac{\pi}{2n}$}
\end{array}\right.
\eeq
with
\[ \delta(\alpha) = \bigg\lfloor\frac{n\alpha}{\pi}\bigg\rfloor\,\frac{\pi}{n}\,. \]  
\end{coro}
$p(0,\alpha)$ is strictly decreasing for $0<\alpha<\frac{\pi}{2n}$, which can be seen as follows: We denote by $f_0^*(\alpha)$ the restriction of $f_0(\alpha)$ to the intervall $[0,\frac{\pi}{n})$. It may be written as
\[ f_0^*(\alpha) = 2\cos^2\frac{\pi}{2n}\bigg[\sin\alpha+\sin\!\bigg(\frac{\pi}{n}-\alpha\bigg)
		+\alpha\cos\!\bigg(\frac{\pi}{n}-\alpha\bigg)+\bigg(\frac{\pi}{n}-\alpha\bigg)\cos\alpha\bigg]\,. \]
One finds
\[ {f_0^*\:\!}'(\alpha)
	= \frac{\dd}{\dd\alpha}\,f_0^*(\alpha)
	= 2\cos^2\frac{\pi}{2n}\,\bigg[\alpha\sin\!\bigg(\frac{\pi}{n}-\alpha\bigg)-\bigg(\frac{\pi}{n}-\alpha\bigg)\sin\alpha\bigg] \]
and hence
\[ \frac{{f_0^*\:\!}'(\alpha)}{\sin\alpha\,\sin\big(\frac{\pi}{n}-\alpha\big)}
	= 2\cos^2\frac{\pi}{2n}\,\bigg(\frac{\alpha}{\sin\alpha}-\frac{\frac{\pi}{n}
	-\alpha}{\sin\big(\frac{\pi}{n}-\alpha\big)}\bigg)\,. \]
For $0<\alpha\leq\frac{\pi}{2n}$, we have $\alpha\leq\frac{\pi}{n}-\alpha$, and therefore,
\[ \frac{\alpha}{\sin\alpha}-\frac{\frac{\pi}{n}-\alpha}{\sin\big(\frac{\pi}{n}-\alpha\big)}\leq 0\,. \]
It follows that ${f_0^*\:\!}'(\alpha)\leq 0$ and hence $p(0,\alpha)\leq 0$ in $0<\alpha\leq\frac{\pi}{2n}$, where the equality signs hold only if $\alpha=\frac{\pi}{2n}$. Due to the symmetry of $p(0,\alpha)$ (see Theorem \ref{Satz2}), $p(0,\alpha)$ is strictly increasing in $\frac{\pi}{2n}<\alpha<\frac{\pi}{n}$. Therefore, the probability of at least one intersection is strictly increasing in $0<\alpha<\frac{\pi}{2n}$ and strictly decreasing in $\frac{\pi}{2n}<\alpha<\frac{\pi}{n}$ (cp.\ \cite{Baesel4}).

Due to its additivity, the expectation $\sum_{i=0}^{2M}\,i\,p(i,\alpha)$ of the number of intersections is always given by $2n(\lambda+\mu)/\pi$.  

\section{Special cases}

The probability of at least one intersection is given by
\[ \frac{2n(\lambda+\mu)}{\pi}\,\sin\frac{\pi}{n}-\frac{n\lambda\mu}{\pi}\,f_0(\alpha)\,. \]
This is one result of Theorem 2.1 in \cite{Baesel4}.

For $\mu=0$ one gets the result for one lattice $\Ra$ of parallel lines in $\mbox{\cite[pp.~17-18]{Baesel1}}$. 

$\mbox{Fig.}\;\ref{p1},\ldots,\ref{p6}$ show diagrams with the intersection probabilities $p(i,\alpha)$, $0\leq\alpha\leq\pi/n$, for $\lambda=1/3$, $\mu=1/4$ and $n=5$.    

\begin{figure}[h]
  \begin{minipage}[b]{.5\linewidth} 
    \includegraphics[scale=0.7]{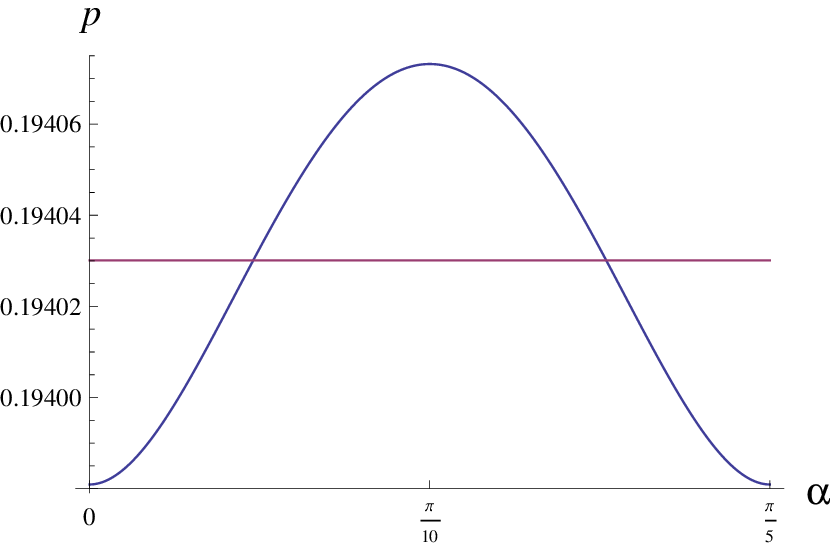}
    \vspace{-0.3cm}
	\caption{\label{p1} $p(1,\alpha)$}
	\vspace{1cm}
  \end{minipage} \vspace{0cm}
  \hspace{0.01\linewidth}
  \begin{minipage}[b]{.5\linewidth} 
    \includegraphics[scale=0.7]{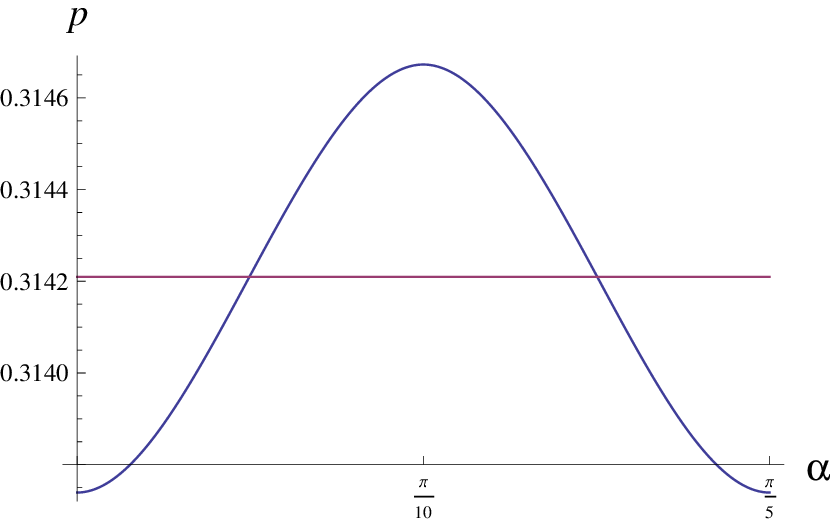}
    \vspace{-0.3cm}
	\caption{$p(2,\alpha)$}
	\vspace{1cm}
  \end{minipage}
  
  \begin{minipage}[b]{.5\linewidth} 
    \includegraphics[scale=0.7]{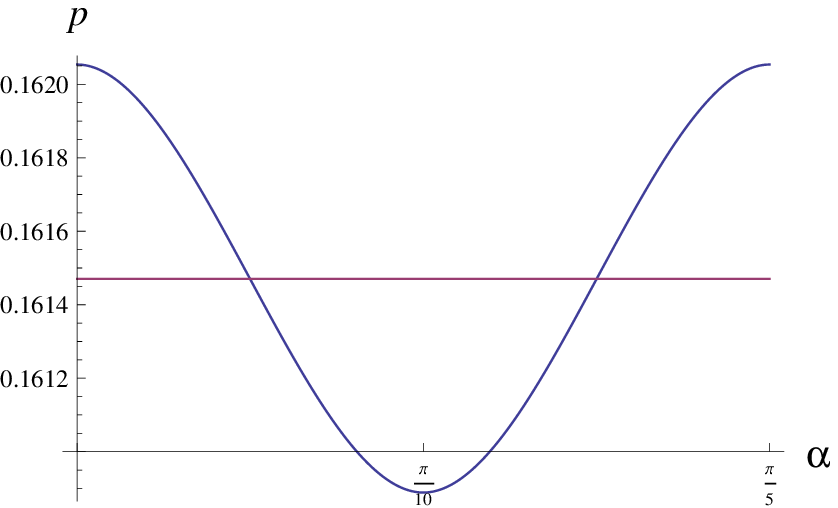}
    \vspace{-0.3cm}
	\caption{$p(3,\alpha)$} 
    \vspace{1.2cm}
  \end{minipage} \vspace{0cm}
  \hspace{0.01\linewidth}
  \begin{minipage}[b]{.5\linewidth} 
    \includegraphics[scale=0.7]{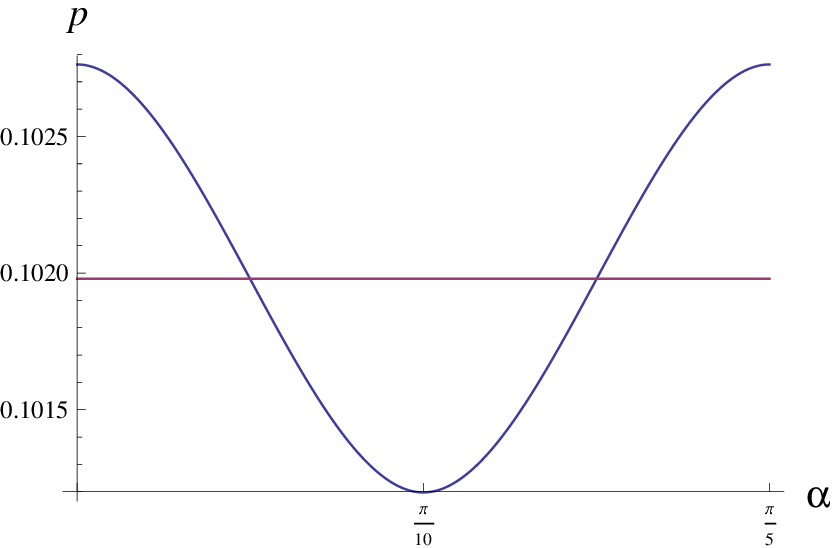}
    \vspace{-0.3cm}
	\caption{$p(4,\alpha)$}
	\vspace{1.2cm}
  \end{minipage}
  
  \begin{minipage}[b]{.5\linewidth} 
    \includegraphics[scale=0.7]{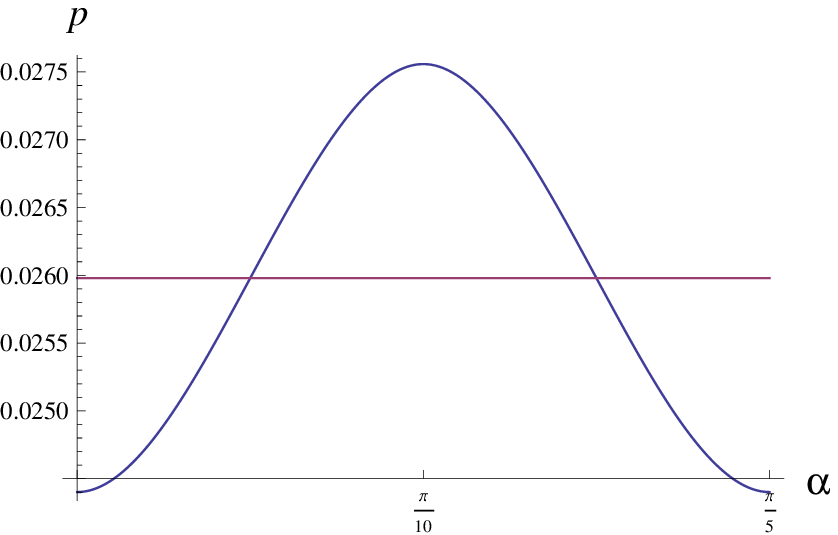}
    \vspace{-0.3cm}
	\caption{$p(5,\alpha)$} 
  \end{minipage} \vspace{0cm}
  \hspace{0.01\linewidth}
  \begin{minipage}[b]{.5\linewidth} 
    \includegraphics[scale=0.7]{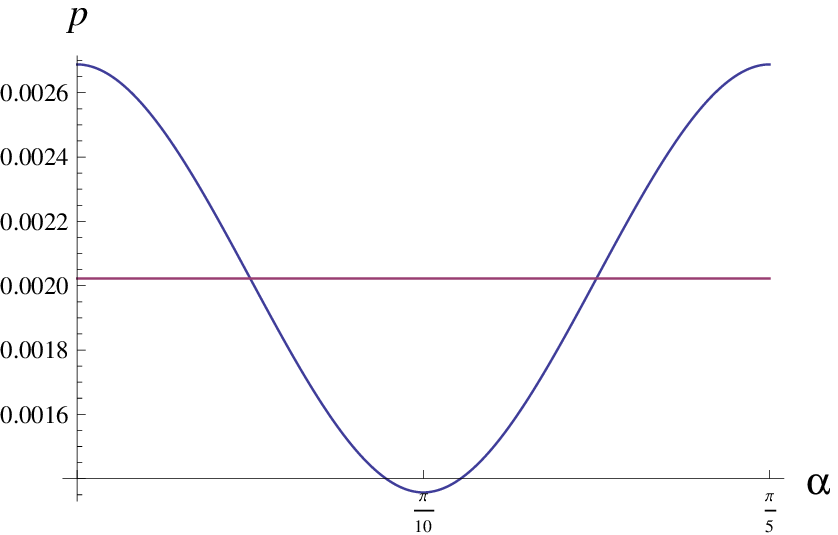}
    \vspace{-0.3cm}
	\caption{\label{p6} $p(6,\alpha)$}
  \end{minipage}
\end{figure}

Using the formulas in Theorem \ref{prob}, we get the following approximate expressions in the case $n=5$:
\begin{align*}
  p(0,\alpha) \approx {} & 1-0,87098(\lambda+\mu)+c_0\lambda\mu\\
  p(1,\alpha) \approx {} & 0,71465(\lambda+\mu)-c_1\lambda\mu\displaybreak[0]\\
  p(2,\alpha) \approx {} & 1,00054(\lambda+\mu)-c_2\lambda\mu\displaybreak[0]\\ 
  p(3,\alpha) \approx {} & 0,155792(\lambda+\mu)+c_3\lambda\mu\displaybreak[0]\\
  p(i,\alpha) \approx {} & c_i\lambda\mu\,,\quad i=4,\,5,\,6, 
\end{align*}
with
\beq
\begin{array}{llll}
  c_0=3,50133\,, & c_1=2,67478\,,  & c_2=3,23888\,, & c_3=0,854102\,,\\
  c_4=1,23316\,, & c_5=0,292814\,, & c_6=0,0322554 
\end{array}  
\eeq
if $\alpha=0,\,\pi/5,\,2\pi/5$, and
\beq
\begin{array}{llll}
  c_0=3,49988\,, & c_1=2,67367\,,  & c_2=3,22768\,, & c_3=0,840122\,,\\
  c_4=1,21437\,, & c_5=0,330696\,, & c_6=0,0162876 
\end{array}  
\eeq
if $\alpha=\pi/10,\,3\pi/10,\,\pi/2$.

From the calculation of many special cases, we conjecture that $p(i,\alpha)$ is strictly increasing in $0<\alpha<\frac{\pi}{2n}$ if $i\in\{1,\ldots,M-1\}$ or $i=2M-1$, and strictly decreasing in this intervall if $i\in\{M,\ldots,2M-2\}$ or $i=2M$.

\clearpage

\section{Distribution functions}

In the following, let $X_{n,\,\alpha}$ denote the ratio
\[ \frac{\mbox{number of intersections between $\Sn$ and $\R$}}{n} \]
(short: {\em relative number of intersections})
and $F_{n,\,\alpha}:\,\mathbb{R}\rightarrow[0,1]$ the distribution function of $X_{n,\,\alpha}$,
\begin{equation*}
  F_{n,\,\alpha}(\xi)\;=\;P(X_{n,\,\alpha}\leq\xi)\;=\;\left\{
  \begin{array}{lll}
	0 & \mbox{for} & -\infty<\xi<0\,,\\[0.1cm]
	\displaystyle{\sum_{i=0}^{\lfloor n\:\!\xi\rfloor}p(i,\alpha)} & \mbox{for} & 0\leq\xi<\dfrac{2M}{n}\,,\\
	1 & \mbox{for} & \dfrac{2M}{n}\leq\xi<\infty\,.
  \end{array} \right.
\end{equation*}
We put
\begin{align*}
  X_{n,\,\lambda} := {} & \frac{\mbox{number of intersections between $\Sn$ and $\Ra$}}{n} \quad\mbox{and}\\
  X_{n,\,\mu}	  := {} & \frac{\mbox{number of intersections between $\Sn$ and $\Rb$}}{n}\,.
\end{align*}
In the case of the independence of $X_{n,\,\lambda}$ and $X_{n,\,\mu}$, the distribution function $F_n$ of $X_n:=X_{n,\,\lambda}+X_{n,\,\mu}$ is given by
\begin{equation*} \label{FXn}
  F_n(\xi)\;=\;P(X_n\leq\xi)\;=\;\left\{
  \begin{array}{lll}
	0 & \mbox{for} & -\infty<\xi<0\,,\\[0.1cm]
	\displaystyle{\sum_{i=0}^{\lfloor n\:\!\xi\rfloor}\sum_{k=0}^i p_\lambda(k)\:p_\mu(i-k)} & \mbox{for} & 0\leq\xi<\dfrac{2M}{n}\,,\\
	1 & \mbox{for} & \dfrac{2M}{n}\leq\xi<\infty\,,
  \end{array} \right.
\end{equation*}
where
\beq
  p_\lambda(i):=\left\{
  \begin{array}{clc}
	p(i,\alpha) \;, & \mbox{if} & 0\leq i\leq M\;,\\[0.2cm]
	0 \;, & \mbox{if} & M+1\leq i\leq 2M\;,
  \end{array}\right.
\eeq
if $\mu=0$ and $\lambda\not=0$, and
\beq
  p_\mu(i):=\left\{
  \begin{array}{clc}
	p(i,\alpha) \;, & \mbox{if} & 0\leq i\leq M\;,\\[0.2cm]
	0 \;, & \mbox{if} & M+1\leq i\leq 2M\;,
  \end{array}\right.
\eeq
if $\lambda=0$ and $\mu\not=0$.

The horizontal lines in the diagrams in $\mbox{Fig.}\;\ref{p1},\ldots,\ref{p6}$ show the values of the probabilities
\[ p^*(i) = \sum_{k=0}^i p_\lambda(k)\:p_\mu(i-k)\,. \]

The question arise if an angle $\alpha$ exist such that $F_n\equiv F_{n,\,\alpha}$. The calculation of many examples shows that it is (in general) not possible to find such a value of $\alpha$ for finite $n$. Therefore, $X_{n,\,\lambda}$ and $X_{n,\,\mu}$ are (in general) dependent random variables.

The random variables $X_{n,\,\lambda}$ and $X_{n,\,\mu}$ converge uniformly to the random variables $X_\lambda$ with distribution function
\beq
  F_\lambda(\xi) = \left\{
  \begin{array}{llc}
	0 & \mbox{for} & -\infty<\xi<0\,,\\[0.15cm]
	1-2\lambda\cos\pi\xi  & \mbox{for} & 0\leq\xi<\frac{1}{2}\,,\\[0.15cm]
	1 & \mbox{for} &  \frac{1}{2}\leq\xi<\infty\,,
  \end{array}\right.  
\eeq
and
$X_\mu$ with distribution function
\beq
  F_\mu(\xi) = \left\{
  \begin{array}{llc}
	0 & \mbox{for} & -\infty<\xi<0\,,\\[0.15cm]
	1-2\mu\cos\pi\xi  & \mbox{for} & 0\leq\xi<\frac{1}{2}\,,\\[0.15cm]
	1 & \mbox{for} &  \frac{1}{2}\leq\xi<\infty\,,
  \end{array}\right.  
\eeq
respectively (see \cite[p. 24]{Baesel1}). If $X_\lambda$ and $X_\mu$ are independent, the distribution of $X_\lambda+X_\mu$ can be calculated with the convolution
\[ F(\xi) = P(X_\lambda+X_\mu\leq\xi) =  \int_{-\infty}^\infty F_\lambda(\xi-\eta)\,\dd F_\mu(\eta)
		\quad\mbox{see \cite[p. 90]{BFWW}}\,, \]    
which yields 
\begin{equation} \label{F}
  F(\xi)=\left\{
  \begin{array}{llc}
	0 & \mbox{for} & -\infty<\xi<0 \,,\\[0.2cm]
	1-2(\lambda+\mu)\cos\pi\xi\\[0.2cm]
	\quad -\,2\lambda\mu(\pi\xi\sin\pi\xi-2\cos\pi\xi) & \mbox{for} & 0\leq\xi<\frac{1}{2}\,,\\[0.2cm]
	1-2\lambda\mu\,\pi(1-\xi)\sin\pi\xi & \mbox{for} & \frac{1}{2}\leq\xi<1\,,\\[0.2cm]
	1 & \mbox{for} & 1\leq\xi<\infty\,,  
  \end{array} \right\}
\end{equation}
$\mbox{(cf. \cite{Baesel5})}$. The following theorem shows that $F$ is not only the distribution function of the sum $X_\lambda+X_\mu$ but also of the random variable $X:=\lim_{n\rightarrow\infty}X_{n,\,\alpha}$. Therefore, $X_{n,\,\lambda}$ and $X_{n,\,\mu}$ are asymptotically independent.

\begin{satz}
As $n\rightarrow\infty$, the random variables $X_{n,\,\alpha}$ converge to the random variable $X$ whose distribution function is given by formula \eqref{F}.
\end{satz}

\begin{proof}
For fixed coordinates $(x,y)$ of the centre point of $\Sn$, the relative number of intersections tends to $\xi=(\sigma+\tau)/(2\pi)$ as $n\rightarrow\infty$ (see Fig.~\ref{Bild4}). The outer parallelogram is the fundamental cell $\F$. $\sigma=\sigma(x,y)$ is the angle of possible intersections with $\Ra$, and $\tau=\tau(x,y)$ the angle of possible intersections with $\Rb$.

\begin{figure}[h]
  \vspace{0cm}
  \begin{center}
    \includegraphics[scale=1]{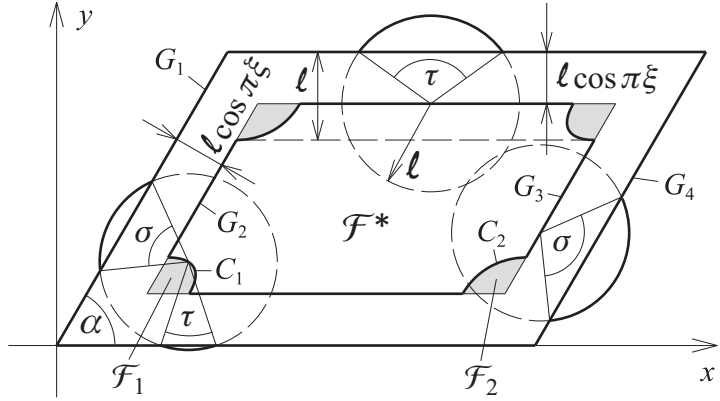}
  \end{center}
  \vspace{-0.2cm}
  \caption{\label{Bild4} Calculation of $F$ for $0\leq\xi<\frac{1}{2}$}
\end{figure}

At first we consider the situation for fixed value of $\xi$ with $0\leq\xi<1/2$. The relative number of intersections is equal to $\xi$ if the centre point of $\Sn$ with $n\rightarrow\infty$ lies on the boundary curve of the set  $\F^*\!\subset\F$; it is $<\xi$ if the centre point lies inside $\F^*$. ($\F^*$ is the inner parallelogram without the four grey coloured sets in its corners.) We denote by $A_1$ and $A_2$ the areas of $\F_1$ and $\F_2$ respectively. Therefore, the limit distribution is given by
\begin{align*}
  F(\xi)
	= {} & \frac{\mbox{Area\;$\F^*(\xi)$}}{\mbox{Area\;$\F$}}
	= \frac{(a-2\ell\cos\pi\xi)(b-2\ell\cos\pi\xi)/\sin\alpha-2(A_1+A_2)}{ab/\sin\alpha}\displaybreak[0]\\[0.2cm]
	= {} & \frac{\big[ab-2\ell(a+b)\cos\pi\xi-4\ell^2\cos^2\pi\xi\big]-2(A_1+A_2)\sin\alpha}{ab}\displaybreak[0]\\
	= {} & 1-2(\lambda+\mu)\cos\pi\xi+4\lambda\mu\cos^2\pi\xi-\frac{2(A_1+A_2)\sin\alpha}{ab}\,.	 
\end{align*}
In the following, we need the equations of the lines $G_1,\ldots,G_4$. They are respectively defined in Hesse normal form by
\begin{align*}
  G_1 = {} & \{(x,y)\in\mathbb{R}^2 \;|\; x\sin\alpha-y\cos\alpha = 0\}\,,\\
  G_2 = {} & \{(x,y)\in\mathbb{R}^2 \;|\; x\sin\alpha-y\cos\alpha-\ell\cos\pi\xi = 0\}\,,\displaybreak[0]\\
  G_3 = {} & \{(x,y)\in\mathbb{R}^2 \;|\; x\sin\alpha-y\cos\alpha-(a-\ell\cos\pi\xi) = 0\}\,,\displaybreak[0]\\
  G_4 = {} & \{(x,y)\in\mathbb{R}^2 \;|\; x\sin\alpha-y\cos\alpha-a = 0\}\,.
\end{align*}
The subset $\F_1\subset\F$ is given by
\[ \F_1 = \{(x,y)\in\mathbb{R}^2\:|\;\ell\cos\pi\xi\leq y\leq\ell\,,\;g_2(y)\leq x\leq f_1(y)\}\,, \]
where
\[ g_2(y) = \frac{1}{\sin\alpha}\,\big(y\cos\alpha + \ell\cos\pi\xi\big) \]
is the equation of $G_2$, and $f_1(y)$ the equation of the curve $C_1$. We get the equation of $C_1$ from
\beq
  \xi = \frac{\sigma+\tau}{2\pi} = \frac{1}{\pi}\,\bigg(\!\arccos\frac{x\sin\alpha-y\cos\alpha}{\ell}
			+ \arccos\frac{y}{\ell}\bigg)
\eeq
which yields
\beqn \label{f1(y)}
  f_1(y) = \frac{1}{\sin\alpha}\bigg[\ell\cos\bigg(\!\pi\xi-\arccos\frac{y}{\ell}\bigg)+y\cos\alpha\bigg]\,;
\eeqn
therefore,
\begin{align*}
  f_1(y)-g_2(y)
	= {} & \frac{\ell}{\sin\alpha}\bigg[\cos\bigg(\!\pi\xi-\arccos\frac{y}{\ell}\bigg)-\cos\pi\xi\bigg]\,.
\end{align*}
So the area of $\F_1$ is given by
\begin{align*}
  A_1
	= {} & \int_{\ell\cos\pi\xi}^\ell\big[f_1(y)-g_2(y)\big]\,\dd y\\
	= {} & \frac{\ell}{\sin\alpha}\,\underbrace{\int_{\ell\cos\pi\xi}^\ell\cos\bigg(\!\pi\xi-\arccos\frac{y}{\ell}\bigg)\,\dd y}_
				{\ds{=:I}}
			\;-\;\frac{\ell\cos\pi\xi}{\sin\alpha}\int_{\ell\cos\pi\xi}^\ell\dd y\,. 
\end{align*}
We calculate the integral $I$. With the substitution $u=y/\ell$, one finds
\begin{align*}
  I	= {} & \ell\int_{\cos\pi\xi}^1\cos(\pi\xi-\arccos u)\,\dd u\displaybreak[0]\\
	= {} & \ell\;\bigg[\cos\pi\xi\int_{\cos\pi\xi}^1\cos(\arccos u)\,\dd u 
				+ \sin\pi\xi\int_{\cos\pi\xi}^1\sin(\arccos u)\,\dd u\bigg]\displaybreak[0]\\
	= {} & \ell\;\bigg[\cos\pi\xi\int_{\cos\pi\xi}^1 u\,\dd u 
				+ \sin\pi\xi\int_{\cos\pi\xi}^1\sqrt{1-u^2}\,\dd u\bigg]\\
	= {} & \frac{\ell}{2}\,\Big[u^2\cos\pi\xi+\Big(u\sqrt{1-u^2}+\arcsin u\Big)\sin\pi\xi\Big]_{\cos\pi\xi}^1
	= \frac{\ell}{2}\,\pi\xi\sin\pi\xi			
\end{align*}
and hence
\[ A_1 = \frac{\ell^2}{2\sin\alpha}\,\big(\pi\xi\sin\pi\xi-2\cos\pi\xi+2\cos^2\pi\xi\big)\,. \]
Now we calculate the area $A_2$ of
\[ \F_2 = \{(x,y)\in\mathbb{R}^2\:|\;\ell\cos\pi\xi\leq y\leq\ell\,,\;f_2(y)\leq x\leq g_3(y)\}\,, \]
where
\[ g_3(y) = \frac{1}{\sin\alpha}\,(a+y\cos\alpha-\ell\cos\pi\xi) \]
is the equation of the line $G_3$, and $f_2(y)$ the equation of the curve $C_2$.
One gets the equation of $C_2$ from
\beq
  \xi = \frac{\sigma+\tau}{2\pi} = \frac{1}{\pi}\,\bigg(\!\arccos\frac{-(x\sin\alpha-y\cos\alpha-a)}{\ell}
			+ \arccos\frac{y}{\ell}\bigg)
\eeq
which gives
\beqn \label{f2(y)}
  f_2(y) = \frac{1}{\sin\alpha}\bigg[a+y\cos\alpha-\ell\cos\bigg(\!\pi\xi-\arccos\frac{y}{\ell}\bigg)\bigg]\,,
\eeqn
and hence
\begin{align*}
  g_3(y)-f_2(y)
	= {} & \frac{\ell}{\sin\alpha}\bigg[\cos\bigg(\!\pi\xi-\arccos\frac{y}{\ell}\bigg)-\cos\pi\xi\bigg] = f_1(y)-g_2(y) \,.
\end{align*}
Due to Cavallieri's principle, we have found that $A_2=A_1$; therefore,
\begin{align*}
  \frac{2(A_1+A_2)\sin\alpha}{ab}
	= {} & \frac{2\ell^2(\pi\xi\sin\pi\xi-2\cos\pi\xi+2\cos^2\pi\xi)}{ab}\\[0.2cm]
	= {} & 2\lambda\mu(\pi\xi\sin\pi\xi-2\cos\pi\xi+2\cos^2\pi\xi)
\end{align*}
and
\begin{align*}
  F(\xi)
	= {} & 1-2(\lambda+\mu)\cos\pi\xi+4\lambda\mu\cos^2\pi\xi-2\lambda\mu\,\big(\pi\xi\sin\pi\xi\\
		 & -2\cos\pi\xi+2\cos^2\pi\xi\big)\\
	= {} & 1-2(\lambda+\mu)\cos\pi\xi-2\lambda\mu\,\big(\pi\xi\sin\pi\xi-2\cos\pi\xi\big)\,.	 
\end{align*}

\begin{figure}[h]
  \vspace{0cm}
  \begin{center}
    \includegraphics[scale=1]{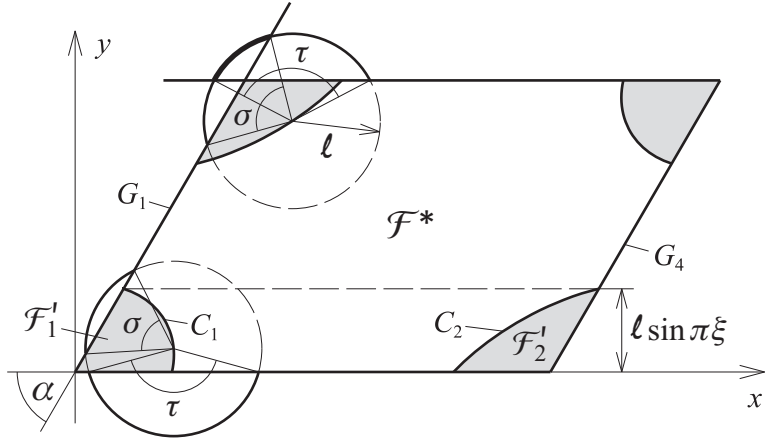}
  \end{center}
  \vspace{-0.2cm}
  \caption{\label{Bild6} Calculation of $F$ for $\frac{1}{2}\leq\xi<1$}
\end{figure}
Now we consider the situation for fixed value of $\xi$ with $\frac{1}{2}\leq\xi<1$ (Fig.~\ref{Bild6}). The parallelogram is the fundamental cell $\F$. $\F^*$ is $\F$ without the four grey coloured sets in its corners. The limit distribution is given by
\begin{align*}
  F(\xi)
	= {} & \frac{\mbox{Area\;$\F^*(\xi)$}}{\mbox{Area\;$\F$}}
	= \frac{ab/\sin\alpha-2(A_1'+A_2')}{ab/\sin\alpha} = 1-\frac{2(A_1'+A_2')\sin\alpha}{ab}\,,
\end{align*}
where $A_1'$ and $A_2'$ are the areas of $\F_1'$ and $\F_2'$ respectively. The subset $\F_1'\subset\F$ is defined by
\[ \F_1' = \{(x,y)\in\mathbb{R}^2\:|\;0\leq y\leq\ell\sin\pi\xi\,,\;g_1(y)\leq x\leq f_1(y)\}\,, \]
where $g_1(y)=y\cot\alpha$ is the equation of $G_1$, and $f_1(y)$ the equation of $C_1$ (see \eqref{f1(y)}). Here the upper limit for the variable $y$ is obtained from
\[ \xi = \frac{1}{2}+\frac{\tau}{2\pi} = \frac{1}{2}+\frac{1}{\pi}\arccos\frac{y}{\ell}
		\quad\Longrightarrow\quad y=\ell\sin\pi\xi \,. \]
So we have
\begin{align*}
  f_1(y)-g_1(y)
	= {} & y\cot\alpha+\frac{\ell}{\sin\alpha}\cos\bigg(\!\pi\xi-\arccos\frac{y}{\ell}\bigg)-y\cot\alpha\\
	= {} & \frac{\ell}{\sin\alpha}\cos\bigg(\!\pi\xi-\arccos\frac{y}{\ell}\bigg)
\end{align*}
and
\begin{align*}
  A_1'
	= {} & \int_0^{\ell\sin\pi\xi}\big[f_1(y)-g_1(y)\big]\,\dd y
	= \frac{\ell}{\sin\alpha}\int_0^{\ell\sin\pi\xi}\cos\bigg(\!\pi\xi-\arccos\frac{y}{\ell}\bigg)\,\dd y\,.			
\end{align*}
Using the substitution $u=y/\ell$, we get
\begin{align*}
  A_1'
	= {} & \frac{\ell^2}{\sin\alpha}\int_0^{\sin\pi\xi}\cos(\pi\xi-\arccos u)\,\dd u\displaybreak[0]\\
	= {} & \frac{\ell^2}{2\sin\alpha}\,\bigg[u^2\cos\pi\xi+u\,\sqrt{1-u^2}\sin\pi\xi
				+\arcsin u\sin\pi\xi\bigg]_0^{\sin\pi\xi}\displaybreak[0]\\
	= {} & \frac{\ell^2}{2\sin\alpha}\,\Big[\sin^2\pi\xi\cos\pi\xi+\sin\pi\xi\,\sqrt{\cos^2\pi\xi}\,\sin\pi\xi
				+\arcsin(\sin\pi\xi)\,\sin\pi\xi\Big]\,.		
\end{align*}
From $\frac{1}{2}\leq\xi<1$, it follows that $\cos\pi\xi\leq 0$ and $\arcsin(\sin\pi\xi)=\pi(1-\xi)$; therefore,
\begin{align*}
  A_1' = {} & \frac{\ell^2}{2\sin\alpha}\,\Big[\!-\!\sin^2\pi\xi\:|\:\!\!\cos\pi\xi|+\sin^2\pi\xi\:|\:\!\!\cos\pi\xi|
				+\pi(1-\xi)\,\sin\pi\xi\Big]\displaybreak[0]\\
	= {} & \frac{\ell^2}{2\sin\alpha}\,\pi(1-\xi)\,\sin\pi\xi\,.			
\end{align*}
The subset $\F_2'\subset\F$ is defined by
\[ \F_2' = \{(x,y)\in\mathbb{R}^2\:|\;0\leq y\leq\ell\sin\pi\xi\,,\;f_2(y)\leq x\leq g_4(y)\}\,, \]
where
\[ g_4(y) = \frac{1}{\sin\alpha}\,(a+y\cos\alpha) \]
is the equation of $G_4$, and $f_2(y)$ the equation of $C_2$ (see \eqref{f2(y)}). We get
\beq
  g_4(y)-f_2(y) = \frac{\ell}{\sin\alpha}\,\cos\bigg(\!\pi\xi-\arccos\frac{y}{\ell}\bigg) = f_1(y)-g_1(y)\,.
\eeq
Due to Cavallieri's principle, we have found that $A'_2=A'_1$. Therefore,
\begin{align*}
  F(\xi) = {} & \frac{ab/\sin\alpha-4A}{ab/\sin\alpha} = 1-\frac{4A\sin\alpha}{ab}
	= 1-2\lambda\mu\,\pi(1-\xi)\,\sin\pi\xi\,,
\end{align*}
and the proof is complete.  			
\end{proof}
Note the interesting fact that the limit distribution $F$ is independent of the angle $\alpha$\,! It is the same limit distribution as for the distribution functions of corresponding {\em clusters of needles} (with equal values of $\lambda$ and $\mu$, respectively) \cite[p.\ 221, Theorem~2]{Baesel5}.\\[0.3cm]
\hspace*{0.4cm} The diagrams in Fig.\ \ref{BildF7} and Fig.\ \ref{BildF25} show for $\lambda=1/3$ and $\mu=1/4$ examples of distribution functions and the limit distribution~$F$.\\[0.3cm]
\hspace*{0.4cm} The calculation of many special cases show (as the diagrams suggest) that it is most likely that the $F_{n,\,\alpha}$ converge {\em uniformly} to $F$. 
\begin{figure}[h]
    \vspace{0.3cm}
  \begin{center}
    \includegraphics[scale=1.1]{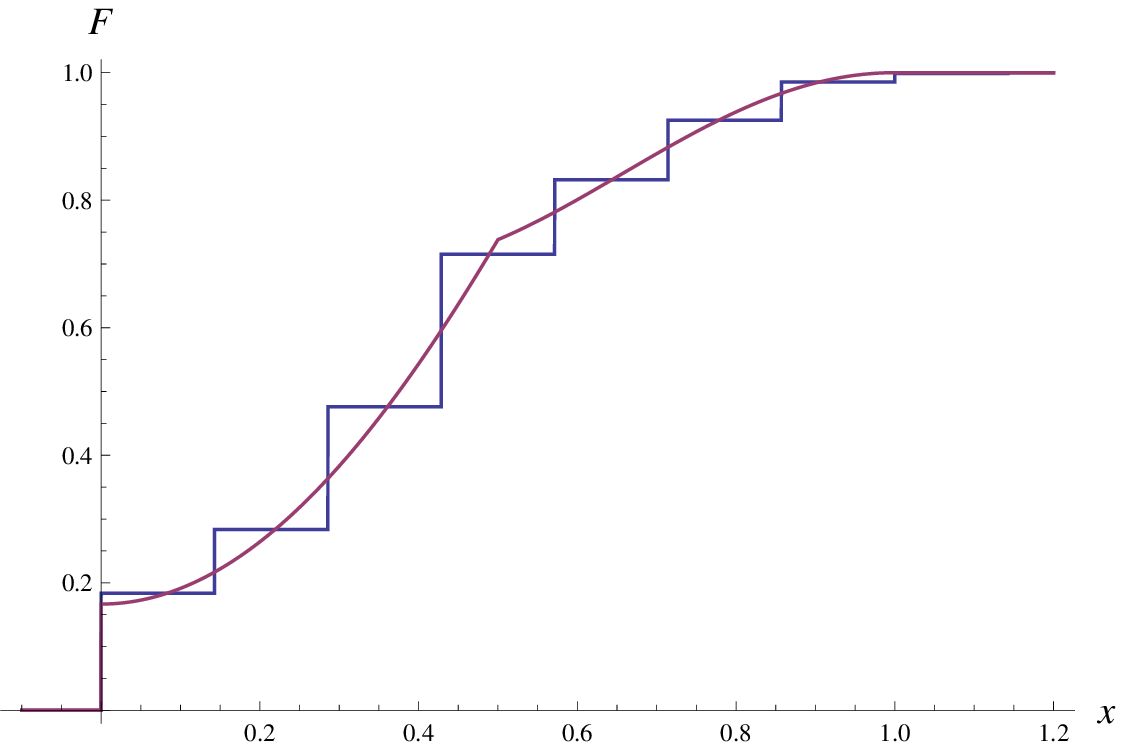}
  \end{center}
  \vspace{-0.3cm}
  \caption{\label{BildF7} $F_{7,\,\alpha}$, $\alpha=k\pi/7\,,\;k=0,1,\ldots,3$, and $F$}
  \vspace{1cm}
  \begin{center}
    \includegraphics[scale=1.1]{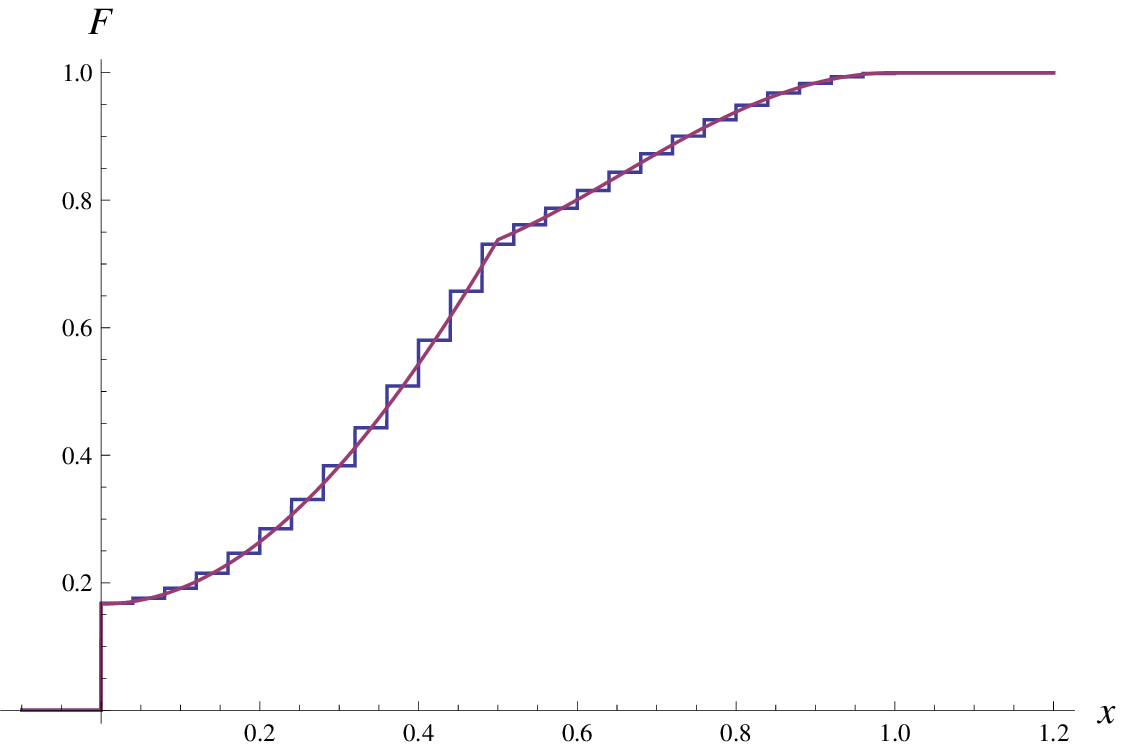}
  \end{center}
  \vspace{-0.3cm}
  \caption{\label{BildF25} $F_{25,\,\alpha}$, $\alpha=k\pi/25\,,\;k=0,1,\ldots,12$, and $F$}
\end{figure}
\clearpage

\vspace{0.2	cm}
\begin{center} 
Uwe B\"asel\\[0.2cm] 
HTWK Leipzig, University of Applied Sciences,\\
Faculty of Mechanical and Energy Engineering,\\
PF 30 11 66, 04251 Leipzig, Germany,\\[0.2cm]
\end{center}
\end{document}